\documentclass[a4paper,12pt]{article}
\usepackage{times, url}
\usepackage{amssymb,latexsym,amsmath,epsfig,amsthm,enumitem,mathtools,xfrac,multicol,amsfonts}

\begin{document}
\setcounter{page}{1}
\noindent{ }
\vspace{20mm}

\begin{center}
{\Large \bf  On the Prime Numbers in the Interval [4n, 5n]}
\vspace{8mm}

{\large \bf Kyle D. Balliet}
\vspace{3mm}

140 Matthew Avenue, Palmerton, PA 18071, USA \\
e-mail: \url{kyleballiet@gmail.com}
\vspace{2mm}

\end{center}
\vspace{10mm}

\noindent
{\bf Abstract:} Is it true that for all $n\geq k\geq 2$ there exists a prime number between $kn$ and $(k+1)n$?  In this paper we show that there is always a prime number between $4n$ and $5n$ for all $n>2$.  We also show there are at least seven prime numbers between $n$ and $5n$ for all $n>5$. \\
\vspace{10mm}

\newtheorem{Theorem}{\noindent Theorem}[section]
\newtheorem{Corollary}[Theorem]{\noindent Corollary}
\newtheorem{Lemma}[Theorem]{\noindent Lemma}
\newtheorem{Question}[Theorem]{\noindent Question}
\newtheorem{Example}[Theorem]{\noindent Example}
\newtheorem{Conjecture}[Theorem]{\noindent Conjecture}

\section{Introduction} 

Bertrand's postulate is stated as:  For $n>1$ there is a prime number between $n$ and $2n$.  This, relatively simply stated, conjecture was stated by Joseph Bertrand in 1845 and was ultimately solved by Pafnuty Chebyshev in 1850, see \cite{2}.  Then in 1932 Paul Erd\H{o}s \cite{3, 4} showed the theorem using elementary properties and approximations of binomial coefficients and the Chebyshev functions:
\[\vartheta(x)=\sum_{p\leq x}\log p,\quad \psi(x)=\sum_{\substack{p^k\leq x\\k\in\mathbb{N}}}\log p.\]

In 2006, Mohamed El Bachraoui \cite{1} showed that for $n>1$ there is a prime number between $2n$ and $3n$.  This similar result to Bertrand's postulated provided a smaller interval for a prime number to exist.  As an example, Bertrand's postulate guarantees $p\in(10,20)$ whereas Bachraoui guarantees $p\in(10,15)$.  Moreover, Bachraoui also questioned if there was a prime number between $kn$ and $(k+1)n$ for all $n\geq k\geq 2$.

Finally, in 2011, Andy Loo \cite{5} shortened these intervals to $p\in(3n,4n)$ for $n>1$.  A. Loo went on to prove that as $n$ approaches infinity the number of prime numbers between $3n$ and $4n$ also tends to infinity - a result which is implied by the prime number theorem.

Throughout this paper we let $p$ be a prime number, $n$ a natural number, $[x]$ as the floor of $x$, and $\log y$ as the logarithm of $y$ with base $e$.

\section{Lemmas}
We will extensively utilize the following three lemmas in the next section and so we present them here.
\begin{Lemma}
For all $x\geq 1$,
\[\sqrt{2\pi}\,x^{x+\frac{1}{2}}e^{-x+\frac{1}{12x+1}}\leq x!\leq \sqrt{2\pi}\,x^{x+\frac{1}{2}}e^{-x+\frac{1}{12x}}.\]
\end{Lemma}
\begin{proof}
See \cite{6}.
\end{proof}

\begin{Lemma}
Let $r$ and $s$ be real numbers satisfying $s>r\geq 1$ and ${s \brace r}=\delta(r,s)\binom{[s]}{[r]}$, then $1\leq \delta(r,s)\leq s$.
\end{Lemma}
\begin{proof}
Let $z>0$ and let $[z]$ be the greatest integer less than or equal to $z$.  Define $\{z\}=z-[z]$.  Let $r$ and $s$ be real numbers satisfying $s>r\geq1$.  Observe that the number of integers in the interval $(s-r,s]$ is $[s]-[s-r]$, which is $[r]$ if $\{s\}\geq\{r\}$ and $[r]+1$ if $\{s\}<\{r\}$.  Let $\mathbb{N}$ be the set of all natural numbers and define
\[{s \brace r}=\frac{\displaystyle\prod_{k\in(s-r,s]\cap\mathbb{N}}\mkern-24mu k}{\displaystyle\prod_{k\in(0,r]\cap\mathbb{N}}\mkern-16mu k}=\delta(r,s)\binom{[s]}{[r]}\]
\noindent where $\delta(r,s)=1$ if $\{s\}\geq\{r\}$ and $\nolinebreak{\delta(r,s)=[s-r]+1}$ if $\{s\}<\{r\}$.  In either case, $1\leq \delta(r,s)\leq s$.
\end{proof}

\begin{Lemma}
The following are true:
\begin{enumerate}[font=\normalfont]
\item Let $c\geq\frac{1}{12}$ be a fixed constant.  For $x\geq\frac{1}{2}$, $\frac{u(x+c)}{l(c)l(x)}$ is increasing.
\item Let $c$ be a fixed positive constant and define $h_2(x)=\frac{u(c)}{l(x)l(c-x)}$.  Then $h'_2(x)>0$ when $\frac{1}{2}\leq x\leq \frac{c}{2}$, $h'_2(x)=0$ when $x=\frac{c}{2}$, and $h'_2(x)<0$ when $\frac{c}{2}<x\leq c-\frac{1}{2}$.
\end{enumerate}
\end{Lemma}
\begin{proof}
See Lemmas 1.4 and 1.5 of \cite{5}.
\end{proof}

\section{Primes between 4n and 5n}
In this section we will show that there is always a prime number in the interval $(4n, 5n)$ for any positive integer $n>2$.  In order to do so, we will begin by showing two inequalities which will be vital to the main proof which follows.

For our main proof we will consider the binomial coefficient $\binom{5n}{4n}$ and subdivide the prime numbers composing $\binom{5n}{4n}$ into three products based on their size.  That is, we will categorize the prime numbers as
\[T_{1}=\prod_{p\leq \sqrt{5n}}\mkern-6mu p^{\beta(p)},\quad T_{2}=\prod_{\sqrt{5n} < p\leq 4n}\mkern-18mu p^{\beta(p)},\quad\text{and } T_{3}=\prod_{4n+1\leq p\leq 5n}\mkern-20mu p\]
so that
\[\binom{5n}{4n}=T_1 T_2 T_3.\]

We will then show that $T_3=\binom{5n}{4n}\frac{1}{T_1 T_2}$ is greater than 1 and therefore there is at least one prime number in $T_3$.  That is, there is at least one prime number in the interval $(4n, 5n)$.

\begin{Lemma}
The following inequalities hold:
\begin{enumerate}[font=\normalfont]
\item For $n\geq 6818$, $e^{\sfrac{1}{(60n+1)}-\sfrac{1}{48n}-\sfrac{1}{12n}}\geq 0.999986$.
\item For $n\geq 1$, $e^{\sfrac{1}{30n}-\sfrac{1}{(24n+1)}-\sfrac{1}{(6n+1)}}\leq 1$.
\item For $n\geq 1$, $e^{\sfrac{1}{20n}-\sfrac{1}{(16n+1)}-\sfrac{1}{(4n+1)}}\leq 1$.
\item For $n\geq 6815$, $\frac{4n+3}{n-3}<4.002202$.
\end{enumerate}
\end{Lemma}
\begin{proof}(1).  The following inequalities are equivalent for $\nolinebreak{n\geq6818}$:
\begin{align*}
e^{\sfrac{1}{(60n+1)}-\sfrac{1}{48n}-\sfrac{1}{12n}}&\geq0.999986\\
\tfrac{1}{60n+1}-\tfrac{1}{48n}-\tfrac{1}{12n}&\geq\log(0.999986)\\
\tfrac{252n+5}{2880n^{2}+48n}&\leq-\log(0.999986)\approx0.000014.
\end{align*}
Now the left-hand side is decreasing in $n$ and so it suffices to verify the case when $n=6818$.  When $n=6818$, we obtain
\begin{align*}
\frac{1718141}{133877484384}&<0.000013\\
&<-\log(0.999986).
\end{align*}

\noindent\textit{Proof} (2).  The following inequalities are equivalent for $n\geq1$:
\begin{align*}
e^{\sfrac{1}{30n}-\sfrac{1}{(24n+1)}-\sfrac{1}{(6n+1)}}&\leq1\\
\tfrac{1}{30n}-\tfrac{1}{24n+1}-\tfrac{1}{6n+1}&\leq0\\
1&\leq756n^{2}+30n.
\end{align*}
Which clearly holds for all $n\geq1$.

\noindent\textit{Proof} (3).  The following inequalities are equivalent for $n\geq1$:
\begin{align*}
e^{\sfrac{1}{20n}-\sfrac{1}{(16n+1)}-\sfrac{1}{(4n+1)}}&\leq1\\
\tfrac{1}{20n}-\tfrac{1}{16n+1}-\tfrac{1}{4n+1}&\leq0\\
1&\leq336n^{2}+20n.
\end{align*}
Which clearly holds for all $n\geq1$.

\noindent\textit{Proof} (4).  The inequality follows directly by noting that $\frac{4n+3}{n-3}<4.002202$ is equivalent to $\nolinebreak{0<0.002202n-15.006606}$ which clearly holds for all $n\geq6815$.
\end{proof}

\begin{Lemma}
For all $n\geq 6818$ the following inequality holds:
\[\frac{0.054886}{2^{\frac{n}{2}}n^{\frac{3}{2}}}{\left(\frac{3125}{256}\right)}^{\frac{n}{6}}>{(5n)}^{\frac{2.51012\sqrt{5n}}{\log(5n)}}.\]
\end{Lemma}
\begin{proof}The following are equivalent for $n\geq 6818$:
\[\frac{0.054886}{2^{\frac{n}{2}}n^{\frac{3}{2}}}{\left(\frac{3125}{256}\right)}^{\frac{n}{6}}>{(5n)}^{\frac{2.51012\sqrt{5n}}{\log(5n)}},\]

\[\log(0.054886)-\frac{n}{2}\log(2)-\frac{3}{2}\text{log}(n)+\frac{n}{6}\log(3125)-\frac{n}{6}\log(256)>2.51012\sqrt{5n},\]

\[\frac{n}{6}{\left[\log(3125)-3\log(2)-\log(256)\right]}>2.51012\sqrt{5n}+\frac{3}{2}\log(n)-\log(0.054886),\]

\[\frac{1}{6}\left[5\log(5)-11\log(2)\right]>\frac{2.51012\sqrt{5}}{\sqrt{n}}+\frac{3}{2}\left(\frac{\log(n)}{n}\right)-\frac{\log(0.054886)}{n}.\]

Now the right-hand side is decreasing in $n$ and so it suffices to verify the case when $n=6818$.  When $n=6818$, we obtain
\[\frac{1}{6}\left[5\log(5)-11\log(2)\right]>0.0742>\frac{2.51012\sqrt{5}}{\sqrt{6818}}+\frac{3}{2}\left(\frac{\log(6818)}{6818}\right)-\frac{\log(0.054886)}{6818}.\qedhere\]
\end{proof}

We now proceed with the proof of our main theorem for this section; that is, there is always a prime number between $4n$ and $5n$ for all integers $n>2$.

\begin{Theorem}
For any positive integer $n>2$ there is a prime number between $4n$ and $5n$.
\end{Theorem}

\begin{proof}It can be easily verified that for $n=3, 4, \ldots, 6817$ there is always a prime number between $4n$ and $5n$.  Now let $n\geq6818$ and consider:
\[\binom{5n}{4n}=\frac{(4n+1)(4n+2)\cdots(5n)}{1\cdot2\cdots n}.\]

The product of primes between $4n$ and $5n$, if there are any, divides $\binom{5n}{4n}$.  Following the notation used in \cite{1,3,5}, we let
\[T_{1}=\prod_{p\leq \sqrt{5n}}\mkern-6mu p^{\beta(p)},\quad T_{2}=\prod_{\sqrt{5n} < p\leq 4n}\mkern-18mu p^{\beta(p)},\quad\text{and } T_{3}=\prod_{4n+1\leq p\leq 5n}\mkern-20mu p\]
such that
\[\binom{5n}{4n}=T_{1}T_{2}T_{3}.\]

The prime decomposition of $\binom{5n}{4n}$ implies that the powers in $T_2$ are less than 2; see \cite[p.~24]{4} for the prime decomposition of $\binom{n}{j}$.  In addition, the prime decomposition of $\binom{5n}{4n}$ yields the upper bound for $T_{1}$:
\[T_{1}<{(5n)}^{\pi(\sqrt{5n})}.\]

But $\pi(x)\leq\frac{1.25506x}{\log(x)}$, see \cite{7}.  So we obtain

\[T_{1}<{(5n)}^{\pi(\sqrt{5n})}\leq{(5n)}^{\frac{2.51012\sqrt{5n}}{\log(5n)}}.\]

\noindent Now let $A={5n/2 \brace 2n}$ and $B={5n/3 \brace 4n/3}$ and observe the following for a prime number $p$ in $T_2$:

\noindent$\circ$ If $\frac{5n}{2}<p\leq 4n$, then
\[n<p\leq 4n<5n<2p.\]
\noindent Hence $\beta(p)=0$.

\noindent$\circ$ Clearly $\displaystyle\prod_{2n<p\leq \frac{5n}{2}}\mkern-12mu p$ divides $A$.

\noindent$\circ$ If $\frac{5n}{3}<p\leq2n$, then
\[n<p<2p\leq4n<5n<3p.\]
\noindent Hence $\beta(p)=0$.

\noindent$\circ$ Clearly $\displaystyle\prod_{\frac{4n}{3}<p\leq \frac{5n}{3}}\mkern-12mu p$ divides $B$.

\noindent$\circ$ If $\frac{5n}{4}<p\leq\frac{4n}{3}$, then
\[n<p<3p\leq4n<5n<4p.\]
\noindent Hence $\beta(p)=0$.

\noindent$\circ$ If $n<p\leq\frac{5n}{4}$, then
\[\frac{n}{2}<p<2n<2p\leq\frac{5n}{2}<3p.\]
\noindent Hence $\displaystyle\prod_{n<p\leq \frac{5n}{4}}\mkern-8mu p$ divides $A$.

\noindent$\circ$ If $\frac{5n}{6}<p\leq n$, then
\[p\leq n<2p<4p\leq 4n<5p\leq 5n<6p.\]
\noindent Hence $\beta(p)=0$.

\noindent$\circ$ If $\frac{2n}{3}<p\leq\frac{5n}{6}$, then
\[\frac{n}{3}<p<\frac{4n}{3}<2p\leq\frac{5n}{3}<3p.\]
\noindent Hence $\displaystyle\prod_{\frac{2n}{3}<p\leq \frac{5n}{6}}\mkern-12mu p$ divides $B$.

\noindent$\circ$ If $\frac{5n}{8}<p\leq\frac{2n}{3}$, then
\[p<n<2p<6p\leq 4n<7p<5n<8p.\]
\noindent Hence $\beta(p)=0$.

\noindent$\circ$ If $\frac{n}{2}<p\leq\frac{5n}{8}$, then
\[\frac{n}{2}<p<3p<2n<4p\leq\frac{5n}{2}<5p.\]
\noindent Hence $\displaystyle\prod_{\frac{n}{2}<p\leq \frac{5n}{8}}\mkern-8mu p$ divides $A$.

\noindent$\circ$ If $\frac{5n}{11}<p\leq\frac{n}{2}$, then
\[p<2p\leq n<3p<8p\leq 4n<9p<10p\leq 5n<11p.\]
\noindent Hence $\beta(p)=0$.

\noindent$\circ$ If $\frac{4n}{9}<p\leq\frac{5n}{11}$, then
\[\frac{n}{3}<p<2p<\frac{4n}{3}<3p<\frac{5n}{3}<4p.\]
\noindent Hence $\displaystyle\prod_{\frac{4n}{9}<p\leq \frac{5n}{11}}\mkern-12mu p$ divides $B$.

\noindent$\circ$ If $\frac{5n}{12}<p\leq\frac{4n}{9}$, then
\[p<2p<n<3p<9p\leq 4n<10p<11p<5n<12p.\]
\noindent Hence $\beta(p)=0$.

\noindent$\circ$ If $\frac{n}{3}<p\leq\frac{5n}{12}$, then
\[\frac{n}{3}<p<3p<\frac{4n}{3}<4p\leq\frac{5n}{3}<5p.\]
\noindent Hence $\displaystyle\prod_{\frac{n}{3}<p\leq \frac{5n}{12}}\mkern-8mu p$ divides $B$.

\noindent$\circ$ If $\frac{5n}{16}<p\leq\frac{n}{3}$, then
\[p<2p<3p\leq n<4p<12p\leq 4n<13p<14p<15p\leq 5n<16p.\]
\noindent Hence $\beta(p)=0$.

\noindent$\circ$ If $\frac{2n}{7}<p\leq\frac{5n}{16}$, then
\[p<\frac{n}{2}<2p<6p<2n<7p<8p\leq\frac{5n}{2}<9p.\]
\noindent Hence $\displaystyle\prod_{\frac{2n}{7}<p\leq \frac{5n}{16}}\mkern-12mu p$ divides $A$.

\noindent$\circ$ If $\frac{5n}{18}<p\leq\frac{2n}{7}$, then
\[p<2p<3p<n<4p<14p\leq 4n<15p<16p<17p<5n<18p.\]
\noindent Hence $\beta(p)=0$.

\noindent$\circ$ If $\frac{n}{4}<p\leq\frac{5n}{18}$, then
\[p<\frac{n}{2}<2p<7p<2n<8p<9p\leq\frac{5n}{2}<10p.\]
\noindent Hence $\displaystyle\prod_{\frac{n}{4}<p\leq \frac{5n}{18}}\mkern-8mu p$ divides $A$.

By the fact that $\displaystyle\prod_{p\leq x}p<4^x$, see \cite[p.~167]{4}, we obtain

\noindent$\circ$  $\displaystyle\prod_{\sqrt{5n}<p\leq \frac{n}{4}}\mkern-12mu p<4^{\frac{n}{4}}=2^{\frac{n}{2}}.$

Now, to summarize, we obtain

\vskip 10pt

\[T_{2}=\displaystyle\prod_{\sqrt{5n} < p\leq 4n}\mkern-14mu p^{\beta(p)}<2^{\frac{n}{2}}AB.\]

\vskip 20pt

By Lemma 2.1, we obtain

\vskip 10pt

\begin{align*}
\binom{5n}{4n}&=\frac{(5n)!}{(4n)!(n!)}\\
&>\frac{l(5n)}{u(4n)u(n)}\\
&=\sqrt{\frac{5}{8\pi n}}{\left(\frac{3125}{256}\right)}^{n}e^{\frac{1}{60n+1}-\frac{1}{48n}-\frac{1}{12n}}\\
&>0.446024n^{-\frac{1}{2}}{\left(\frac{3125}{256}\right)}^{n},
\end{align*}

\vskip 20pt

\noindent where the last inequality follows from the fact that $e^{\frac{1}{60n+1}-\frac{1}{48n}-\frac{1}{12n}}\geq 0.999986$ for all $n\geq6818$ by Lemma 3.1.

\vskip 20pt

Similarly, by Lemmas 2.1, 2.2, and 2.3, we obtain

\vskip 10pt

\begin{align*}
A&={\frac{5n}{2} \brace 2n}\leq \frac{5n}{2}\binom{[\frac{5n}{2}]}{2n}\\
&=\frac{5n}{2}\cdot\frac{[\frac{5n}{2}]!}{(2n)!([\frac{5n}{2}]-2n)!}\\
&<\frac{5n}{2}\cdot\frac{u([\frac{5n}{2}])}{l(2n)l([\frac{5n}{2}]-2n)}\\
&\leq\frac{5n}{2}\cdot\frac{u(\frac{5n}{2})}{l(2n)l(\frac{n}{2})}\\
&=\frac{5}{4}\sqrt{\frac{5n}{\pi}}{\left(\frac{3125}{256}\right)}^{\frac{n}{2}}e^{\frac{1}{30n}-\frac{1}{24n+1}-\frac{1}{6n+1}}\\
&<\frac{5}{4}\sqrt{\frac{5n}{\pi}}{\left(\frac{3125}{256}\right)}^{\frac{n}{2}}\\
&<1.576958n^{\frac{1}{2}}{\left(\frac{3125}{256}\right)}^{\frac{n}{2}},
\end{align*}
and finally
\begin{align*}
B&={\frac{5n}{3} \brace \frac{4n}{3}}\leq\frac{5n}{3}\binom{[\frac{5n}{3}]}{[\frac{4n}{3}]}\\
&=\frac{5n}{3}\cdot\frac{[\frac{4n}{3}]+1}{[\frac{5n}{3}]-[\frac{4n}{3}]}\cdot\binom{[\frac{5n}{3}]}{[\frac{4n}{3}]+1}\\
&\leq\frac{5n}{3}\cdot\frac{4n+3}{n-3}\cdot\frac{u(\frac{5n}{3})}{l(\frac{4n}{3})l(\frac{n}{3})}\\
&=\sqrt{\frac{125}{24\pi}}{\left(\frac{4n+3}{n-3}\right)}n^{\frac{1}{2}}{{\left(\frac{3125}{256}\right)}^{\frac{n}{3}}}\\
&\quad\quad\quad\cdot e^{\frac{1}{20n}-\frac{1}{16n+1}-\frac{1}{4n+1}}\\
&<5.153158n^{\frac{1}{2}}{{\left(\frac{3125}{256}\right)}^{\frac{n}{3}}},
\end{align*}
where the inequalities follow by Lemma 3.1.

Thus we obtain
\begin{align*}
{T_3}&=\binom{5n}{4n}\frac{1}{{T_2}{T_1}}\\
&>\binom{5n}{4n}\frac{1}{2^{\frac{n}{2}}AB}\cdot\frac{1}{{(5n)}^{\frac{2.51012\sqrt{5n}}{\log(5n)}}}\\
&>\frac{0.054886}{2^{\frac{n}{2}}n^{\frac{3}{2}}}{\left(\frac{3125}{256}\right)}^{\frac{n}{6}}\cdot\frac{1}{{(5n)}^{\frac{2.51012\sqrt{5n}}{\log(5n)}}}>1,
\end{align*}
where the last inequality follows by Lemma 3.2.  Consequently, the product $T_3$ of prime numbers between $4n$ and $5n$ is greater than 1 and therefore the existence of such numbers is proven.
\end{proof}

\section{Consequences}
With the proof of the previous theorem complete we may also show that there is always a prime number between $n$ and $(5n+15)/4$ for all positive integers $n>2$ as in the following theorem.

\begin{Theorem}
For any positive integer $n>2$ there exists a prime number $p$ satisfying $n<p<\frac{5(n+3)}{4}$.
\end{Theorem}
\begin{proof}When $n=3$, we obtain $3<5<7<\frac{15}{2}$.  Let $n\geq 4$.  By the division algorithm $4\mid(n+r)$ for some $r\in\{0,1,2,3\}$ and by Theorem 3.3 there exists a prime number $p$ such that $p\in\left(n+r,\frac{5(n+r)}{4}\right)$.

Since $\left(n+r,\frac{5(n+r)}{4}\right)$ is contained in $\left(n,\frac{5(n+3)}{4}\right)$ for all $0\leq r\leq 3$ and $n>2$,\\ $\nolinebreak{p\in\left(n,\frac{5(n+3)}{4}\right)}$ as desired.
\end{proof}
\newpage
\begin{Theorem}
For any positive integer $n>2$ there are at least four prime numbers between $n$ and $5n$.
\end{Theorem}
\begin{proof}The cases when $n=3, 4, \ldots, 14$ may be verified directly.  Now let $n\geq15$ and by Theorem 4.1 we know there exists prime numbers $p_1$, $p_2$, and $p_3$ such that $n<p_1<\frac{5n+15}{4}$, $2n<p_2<\frac{10n+15}{4}$, $3n<p_3<\frac{15n+15}{4}$, and by Theorem 3.3 there exist a prime number $p_4$ such that $4n<p_4<5n$.  Hence
\[n<p_1<\frac{5n+15}{4}<2n<p_2<\frac{10n+15}{4}<3n<p_3<\frac{15n+15}{4}\leq4n<p_4<5n.\qedhere\]
\end{proof}

In the next theorem we improve on the number of prime numbers between $n$ and $5n$ to show that there are at least seven prime numbers between $n$ and $5n$ for $n>5$.

\begin{Theorem}
For all $n>5$ there are at least seven prime numbers between $n$ and $5n$.
\end{Theorem}
\begin{proof}
The cases when $n=6,7,\ldots, 244$ may be verified directly.  Now let $f(n)=\frac{5n+15}{4}$ for $n\geq 245$ and let $f^m(n)=f(f^{m-1}(n))$.  By Theorem 4.1 there exists a prime number between $n$ and $f(n)$.  Furthermore, there exists a prime number between $f^{m-1}(n)$ and $f^m(n)$ for all $\nolinebreak{m\in\mathbb{N}\setminus\{1\}}$.  In general,
\[f^m(n)=\frac{1}{4^m}\left(5^m n+3\displaystyle\sum_{k=0}^{m-1}5^{m-k}\,4^k\right).\]

Consider
\[f^m(n)=\frac{1}{4^m}\left(5^m n+3\displaystyle\sum_{k=0}^{m-1}5^{m-k}\,4^k\right)\leq 5n.\]

Solving for $n$, we obtain
\[n\geq\frac{3}{5\cdot 4^m-5^m}\displaystyle\sum_{k=0}^{m-1}5^{m-k}\,4^k.\]

However, $5\cdot 4^m-5^m$ is positive only for $m\leq 7$.  So let $m=7$, then for
\[n\geq 245>\frac{3}{5\cdot 4^7-5^7}\sum_{k=0}^{6}5^{7-k}\,4^k\]
there are at least seven prime numbers between $n$ and $5n$ and our proof is complete.
\end{proof}

\begin{Theorem}
For $n>2$ the number of prime numbers in the interval $(4n, 5n)$ is at least
\[\log_{5n}\left[\frac{0.054886}{2^{\frac{n}{2}}n^{\frac{3}{2}}}{\left(\frac{3125}{256}\right)}^{\frac{n}{6}}{(5n)}^{-\frac{2.51012\sqrt{5n}}{\log(5n)}}\right].\]
\end{Theorem}
\begin{proof}
In Theorem 3.3 we approximated the product of prime numbers between $4n$ and $5n$ from below by
\[\frac{0.054886}{2^{\frac{n}{2}}n^{\frac{3}{2}}}{\left(\frac{3125}{256}\right)}^{\frac{n}{6}}{(5n)}^{-\frac{2.51012\sqrt{5n}}{\log(5n)}}.\]
Bounding each of the prime numbers between $4n$ and $5n$ from above by $5n$, we obtain:
\begin{align*}
&\log_{5n}\left[\frac{0.054886}{2^{\frac{n}{2}}n^{\frac{3}{2}}}{\left(\frac{3125}{256}\right)}^{\frac{n}{6}}{(5n)}^{-\frac{2.51012\sqrt{5n}}{\log(5n)}}\right]\\
&\qquad\qquad=\frac{\log 0.054886-\frac{n}{2}\log 2+\frac{n}{6}\log \frac{3125}{256}}{\log n+\log 5}+\frac{\frac{3}{2}\log 5-2.51012\sqrt{5n}}{\log n+\log 5}-\frac{3}{2}\\
&\qquad\qquad>\frac{n\left(\frac{1}{6}\log \frac{3125}{256}-\frac{1}{2}\log 2-\frac{2.51012\sqrt{5}}{\sqrt{n}}\right)}{2\log n}+\frac{\frac{3}{2}\log 5+\log 0.054886}{\log n+\log 5}-\frac{3}{2}\\
&\qquad\qquad>\frac{n}{\log n}\left(0.035214-\frac{2.8063995}{\sqrt{n}}\right)-\frac{168033}{100000}.
\end{align*}

Now observe that $\displaystyle\lim_{n\to\infty}\frac{2.8063995}{\sqrt{n}}=0$.  Moreover $\displaystyle\lim_{n\to\infty}\frac{n}{\log n}=+\infty$.
\end{proof}

By Theorem 4.4 we have the following theorem.

\begin{Theorem}
As $n$ tends to infinity, the number of prime numbers in the interval $[4n, 5n]$ goes to infinity.  That is, for every positive integer $m$, there exists a positive integer $L$ such that for all $n\geq L$, there are at least $m$ prime numbers in the interval $[4n, 5n]$.
\end{Theorem}
 
\makeatletter
\renewcommand{\@biblabel}[1]{[#1]\hfill}
\makeatother


\begin{thebibliography}{99}


\bibitem{1} Bachraoui, M. El.,  Primes in the interval [2n,3n], \textit{Int. J. Contemp. Math. Sci.}, Vol. 1, 2006, 617--621.

\bibitem{2} Chebyshev, P.,  M\'{e}moire sur les nombres premiers, \textit{M\'{e}m. Acad. Sci. St. Ptersbourg}, Vol. 7, 1850, 17--33.

\bibitem{3} Erd\H{o}s, P.,  Beweis eines satzes von tschebyschef, \textit{Acta Litt. Univ. Sci., Szeged, Sect. Math.}, Vol. 5, 1932, 194--198.

\bibitem{4} Erd\H{o}s, P., J. Sur\'{a}nyi, \textit{Topics in the Theory of Numbers}, Springer Verlag, 2003.

\bibitem{5} Loo, A.,  On the primes in the interval [3n,4n], \textit{Int. J. Contemp. Math. Sci.}, Vol. 6, 2011, 1871--1882.

\bibitem{6} Robbins, H.,  A remark on Stirling's formula, \textit{Amer. Math. Monthly}, Vol. 62, 1955, 26--29.

\bibitem{7} Rosser, J., L. Schoenfeld, Approximate formulas for some functions of prime numbers, \textit{Illinois J. Math.}, Vol. 6, 1962, 64--94.





\end{thebibliography}
\end{document}